\documentclass[12pt, ams fonts]{amsart}

\usepackage{color}
\usepackage{stmaryrd}
\usepackage{tikz-cd}
\usepackage{mathtools}
\usepackage{amssymb,amsmath}
\usepackage{amsfonts}
\usepackage{graphicx}
\usepackage{nicematrix}

\usepackage[vcentermath,enableskew]{youngtab}

\input xy
\xyoption{all}

\font\teneufm=eufm10 \font\seveneufm=eufm7 \font\fiveeufm=eufm5
\newfam\eufmfam
\textfont\eufmfam=\teneufm \scriptfont\eufmfam=\seveneufm
\scriptscriptfont\eufmfam=\fiveeufm

\newcommand{\C}{\mathbb{C}}

\newcommand{\Z}{\mathbb{Z}}

\newcommand{\h}{\mathfrak{h}}

\DeclareMathOperator{\Span}{Span}

\DeclareMathOperator{\Id}{Id}

\DeclareMathOperator{\Mat}{Mat}
\DeclareMathOperator{\GL}{GL}
\DeclareMathOperator{\Res}{Res}
\DeclareMathOperator{\I}{I}

\DeclareMathOperator{\End}{End}

\DeclareMathOperator{\Comp}{Comp}

\numberwithin{equation}{section}

\newtheorem{definition}{Definition}[section]

\newtheorem{example}[definition]{Example}

\theoremstyle{remark}

\newtheorem{remark}[definition]{Remark}

\theoremstyle{plain} 

\newtheorem{theorem}[definition]{Theorem}
\newtheorem{lemma}[definition]{Lemma}
\newtheorem{corollary}[definition]{Corollary}
\newtheorem{proposition}[definition]{Proposition}

\def\Z{\mathbb Z}
\def\C{\mathbb C}

\begin{document}

\title{Exponential modules of $\displaystyle \mathfrak{osp}(1|2)$}
\author[D. Grantcharov]{Dimitar Grantcharov}
	\address{D. Grantcharov: Department of Mathematics, University of Texas at Arlington}
	\email{grandim@uta.edu}	
\author[K. Nguyen]{Khoa Nguyen}
	\address{K. Nguyen: Department of Mathematics and Statistics, Queen's University, Kingston, ON K7L 3N6, Canada}
	\email{k.nguyen@queensu.ca}

\begin{abstract}
We study properties of two  families, $E_{+}(g)$ and $E_-(g)$, of non-weight modules over the orthosymplectic Lie superalgebra $\mathfrak{osp}(1|2)$ that are parameterized by a nonconstant polynomial $g(x) \in \mathbb C [x]$. These families appear naturally from the two oscillator homomorphisms and the exponential modules over the first Weyl algebra $\mathcal D(1)$.  
 We prove simplicity and isomorphism theorems for $E_{+}(g)$ and $E_-(g)$.
\end{abstract}
\subjclass[2020]{17A70, 17B10}
	
	\keywords{orthosymplectic Lie superalgebra, $U(\mathfrak{h})$-free modules, Weyl algebra,  oscillator representation, exponential modules}

\maketitle

\section{Introduction}
Modules over Lie superalgebras are fundamental in both mathematics and physics. The orthosymplectic  superalgebras $\mathfrak{osp}(1|2n)$ are special in that their universal enveloping algebras are domains, which makes their representation theory easier to study. Like all other  Lie (super)algebras, the representation theory of $\mathfrak{osp}(1|2n)$ splits into various categories. Foremost is the category of \emph{weight modules}, whose objects admit a direct sum decomposition into $\mathfrak{h}$-weight spaces for a fixed Cartan subalgebra $\mathfrak{h}$. Weight modules of these orthosymplectic superalgebras are studied in \cite{GGF}, and of their direct limits in \cite{GPS}. Among non-weight modules, recently, one category attracted considerable attention — the category $\mathcal{FR}(n)$ of $U(\mathfrak{h})$-free modules of finite rank $n$. The case of rank $1$ is relatively well understood as the classification of all objects in $\mathcal{FR}(1)$ is established for $\mathfrak{sl}(n{+}1)$ in \cite{Nil1}, for the algebras of Witt type in \cite{TZ1}, \cite{TZ2}, and for $\mathfrak{osp}(1|2n)$ in \cite{CZ}. The case of arbitrary rank is challenging already for $\mathfrak{sl}(2)$, which will be treated in detail in a forthcoming paper.

One special type of modules in $\mathcal{FR}(n)$ are the so-called exponential modules — those that are pull-backs of exponential-type modules over the Weyl algebras $\mathcal D(n)$. In the case of $\mathfrak{sl}(n{+}1)$, it was shown in \cite{GN} that all objects in $\mathcal{FR}(1)$ are tensor exponential modules. Exponential modules for the Smith algebra were investigated in \cite{VSM}.

In this paper we initiate the systematic study of exponential modules over the orthosymplectic superalgebras by looking first  at the case of $\mathfrak{osp}(1|2)$. It is worth noting that there is an abstract classification of all simple modules over $\mathfrak{osp}(1|2)$ in \cite{BO}, but this classification is not explicit and it relates to irreducible elements of a certain Euclidean ring. In contrast, the modules we study are two explicit families of simple $U(\mathfrak h)$-free modules of $\mathfrak{osp}(1|2)$ of arbitrary finite rank that come from a natural $D$-module setting. The starting point is the first Weyl algebra $\mathcal D(1)$ and the two classical oscillator homomorphisms
\[
\Phi_{\pm}:U(\mathfrak{osp}(1|2))\longrightarrow \mathcal D(1),
\]
cf.\ \S\ref{sec:prelim} and \cite{GGF,P}. For any nonconstant $g(x)\in\C[x]$ we consider the exponential $\mathcal D(1)$-module $E(g)=\C[x]e^{g(x)}$ and pull it back along $\Phi_\pm$ to obtain $U(\mathfrak{osp}(1|2))$–modules $E_\pm(g)$. In particular, both $E_\pm(g)$ have fixed $\mathfrak{osp}(1|2)$- and $\mathfrak{sl}(2)$-central characters. When $g$ is even, $E_\pm(g)$ carries a natural $\Z_2$–grading. We summarize our main results below.

\medskip
\noindent\textbf{Main results.}
Let $g(x)=\sum_{j=1}^{n}a_jx^j$ with $n=\deg g\geq 1$ and $a_n\neq 0$.
\begin{itemize}
  \item \emph{$U(\mathfrak h)$-freeness.} We prove
  \[
  E_{\pm}(g)\ \simeq\ \bigoplus_{p=0}^{n-1}\,U(\mathfrak h)\cdot (x^{p}e^{g}),
  \]
  hence $E_{\pm}(g)$ are $U(\mathfrak h)$-modules of rank $n$ (Proposition~\ref{U(h)-decomp}).
  \item \emph{Simplicity.} Both $E_{\pm}(g)$ are simple $U(\mathfrak{osp}(1|2))$-modules (Proposition~\ref{simplicity}).
  \item \emph{Isomorphism classification.} If $s\in\{+,-\}$, then, after the normalization $g_i(0)=0$,
  \[
  E_{s}(g_1)\simeq E_{s}(g_2)\quad\Longleftrightarrow\quad g_1=g_2\ 
  \]
  (Theorem~\ref{thm:isom}). Moreover, $E_{+}(g_1)\not\simeq E_{-}(g_2)$ unless $\deg g_1=\deg g_2=2$, in which case the two are isomorphic precisely in the \emph{integral Fourier-dual} situation, i.e., when $b_2=-1/(4a_2)$ and $b_1=-a_1/(2a_2)$ for $g_1(x)=a_1x+a_2x^2$, $g_2(x)=b_1x+b_2x^2$.
  \item \emph{Explicit $F(X_+,X_-)$–realization.} If $s\in\{+,-\}$, we give companion-type matrices $X_\pm(h)\in\Mat_n(\C[h])$ such that $E_{s}(g)\simeq F(X_+,X_-)$ (Theorem~\ref{explicitrealization}). 
  \item \emph{Concrete bases and closed formulas.} Choosing the $U(\mathfrak h)$–basis $\{x^p e^{g}\}_{p=0}^{n-1}$ yields explicit isomorphisms
  \(
  \Psi_{\pm}:E_{\pm}(g)\xrightarrow{\ \sim\ }\C[h]^{\oplus n}.
  \)
  We describe $\Psi_{\pm}(x^p e^{g})$ both via generating-function ODEs and through closed combinatorial formulas; see \S\ref{subsec:ODE}–\S\ref{subsec:combinatorics}.
\end{itemize}

The paper is organized as follows. Section~\ref{sec:prelim} collects the needed background on $\mathcal D(1)$, the oscillator maps $\Phi_{\pm}$, and central characters. In Section \ref{sec:FRn-matrix} we set up the matrix realization $F(X_+,X_-)$. Section~\ref{sec:exp-mods} introduces the exponential modules $E_{\pm}(g)$ and proves $U(\mathfrak h)$-freeness, simplicity, and the isomorphism theorems. The matrix realization of $E_{\pm}(g)$ is provided in Theorem~\ref{explicitrealization}. Finally, \S\ref{subsec:ODE} and \S\ref{subsec:combinatorics} provide explicit bases and formulas for the $\Psi_{\pm}$–isomorphisms via ODEs and recursive sequences.

\medskip
\noindent{\it Acknowledgements.}  We would like to thank Maria Gorelik for the useful suggestions. D.G. is supported in part by Simons Collaboration 
Grant 855678.

\section{Notation and Conventions}
Throughout this paper, we denote the sets of integers, complex numbers, and nonzero complex numbers by $\Z$, $\C$, and $\C^*$, respectively. For any integer $k$, we write $\Z_{\geq k}$ for the set $\{i\in \Z: i\geq k\}$. Given a Lie (super)algebra $\mathfrak{g}$, its universal enveloping algebra is denoted $U(\mathfrak g)$, and we fix a Cartan subalgebra $\mathfrak h\subset \mathfrak g$. For a ring $R$, we write $\Mat_N(R)$ for the ring of $N\times N$ matrices with entries in $R$. 

All vector spaces and algebras are over $\C$ unless otherwise stated. By $\mathcal D(1)$ we denote the first Weyl algebra, i.e. the algebra of polynomial differential operators on $\C[x]$.  In terms of generators and relations,  $\mathcal D(1)$ is generated by $x$ and $\partial = \partial_x$, subject to the relation
$$
\partial x - x \partial = 1.
$$
By $\mathcal O$ we denote the natural module $\mathbb C[x]$ of $\mathcal D(1)$.

For a module $M$ over an associative unital algebra $A$ and an automorphism $f \in \mbox{Aut}(A)$, by $M^f$ we denote the twisted by $f$ module. Namely, $M^f = M$ as vector spaces, but with $A$-action defined by $a\cdot_f m = f(a)m$.

We define the automorphism
$\sigma:\mathbb{C}[h]\;\to\;\mathbb{C}[h]$ by
$$
\sigma\bigl(f(h)\bigr)\;:=\;f(h-1),
\qquad\text{for every }f(h)\in\mathbb{C}[h].
$$
 For a complex number or polynomial $\alpha$ and $k\in\mathbb Z_{\ge0}$,
\[
\binom{\alpha}{k}:=\frac{\alpha(\alpha-1)\cdots(\alpha-k+1)}{k!},\qquad \binom{\alpha}{0}:=1.
\]

\section{Preliminaries on $\mathfrak{osp}(1|2)$ and $\mathcal D(1)$} \label{sec:prelim}

\subsection{Two automorphisms of ${\mathcal D} (1)$}
We will use two automorphisms of ${\mathcal D} (1)$ - the Fourier transform and the exponentiation.

The \emph{Fourier transform} $\theta$ is the automorphism of ${\mathcal D} (1)$ defined as follows:
\begin{eqnarray*}
\theta(x) = \partial,\quad \theta(\partial) = -x.
\end{eqnarray*}

For any $g \in \C[x]$, the \emph{$g$-exponentiation} $\epsilon_g$ on $\mathcal D(1)$ is defined by: $\epsilon_{g} (x) = x$,  $\epsilon_{g} (\partial) = \partial+ g'(x)$.
Since $\epsilon_{g} = \epsilon_{g+c}$, for any $c \in \mathbb C$, \emph{we will assume that $g(0)=0$ whenever $\epsilon_{g}$ is considered}. Sometimes $\epsilon_g$ is called the shear automorphism defined by $g'(x)$.

\subsection{Basis of  $\mathfrak{osp} (1|2)$}
Throughout the paper, we fix a Cartan subalgebra $\mathfrak h = \mathbb C h$  and a root system $\{ \delta, -\delta, 2\delta, -2\delta\}$ of $\mathfrak{osp} (1|2)$ so that: 
$$\mathfrak{osp} (1|2) = \mathfrak{osp}(1|2)_{\bar{0}} \oplus \mathfrak{osp}(1|2)_{\bar{1}}$$
with
$$\mathfrak{osp}(1|2)_{\bar{0}} = \Span\{h,x_{2\delta},x_{-2\delta}\},\quad \mathfrak{osp}(1|2)_{\bar{1}} = \Span\{x_{\delta},x_{-\delta}\} $$
for which the following superbracket relations hold:
$$
\begin{aligned}
&[h,x_{2\delta}]=2x_{2\delta}, \qquad [h,x_{-2\delta}]=-2x_{-2\delta}, \qquad [x_{2\delta},x_{-2\delta}]=h,\\
&[h,x_{\pm\delta}] = \pm x_{\pm\delta}, \qquad [x_{2\delta},x_{-\delta}] = -\,x_{\delta}, \qquad [x_{-2\delta},x_{\delta}] =- x_{-\delta},\\
&[x_{2\delta},x_{\delta}] = [x_{-2\delta},x_{-\delta}] = 0,\\
&[x_{\delta},x_{\delta}]=2x_{2\delta},\qquad [x_{-\delta},x_{-\delta}]=-2x_{-2\delta},\qquad [x_{\delta},x_{-\delta}]=h.
\end{aligned}
$$

We will identify  $U(\mathfrak{h})$ with $ \C[h]$. In particular, we will often use the following identities:
$$
x_{\delta} f(h) = f(h-1) x_{\delta},\quad x_{-\delta} f(h) = f(h+1) x_{-\delta}  
$$
in $U(\mathfrak{osp}(1|2))$ for any  $f(h) \in \C [h]$. 

The associative algebra $U(\mathfrak{osp}(1|2))$ is generated by $x_{\delta}, x_{-\delta}, h$ subject to the three relations
\begin{equation}\label{rel-osp}
x_{\delta}x_{-\delta} + x_{-\delta}x_{\delta} =h, \; hx_{\delta} - x_{\delta}h = x_{\delta},\; hx_{-\delta} - x_{-\delta}h = -x_{-\delta}.
\end{equation}
\subsection{The centers of $U(\mathfrak{osp}(1|2))$ and $U(\mathfrak{osp}(1|2))_{\bar{0}}$}
We have two Casimir elements - the Casimir element $C$ of $\mathfrak{osp}(1|2)_{\bar{0}}$, and  the Casimir element $\Omega$ of $\mathfrak{osp}(1|2)$  defined by
$$
C:=x_{-2\delta}x_{2\delta}+\frac{1}{2}h+\frac{1}{4}h^2
   =x_{2\delta}x_{-2\delta}-\frac{1}{2}h+\frac{1}{4}h^2,
\qquad
\Omega:=C-\frac{1}{4}\bigl(x_{\delta}x_{-\delta}-x_{-\delta}x_{\delta}\bigr).
$$
More precisely, the center of $Z(\mathfrak{osp}(1|2)_{\bar 0})$ of $U(\mathfrak{osp}(1|2)_{\bar 0})$ coincides with the polynomial algebra $\C[C]$, while the center of $U(\mathfrak{osp}(1|2))$  is $Z(\mathfrak{osp}(1|2)) = \C[\Omega]$. A direct computation shows that
$$
4C^2 - (8\Omega - 1)C + 2\Omega(2\Omega - 1) = 0.
$$
For details we refer the reader to \cite{P}. Note that the basis $\{F,G,Y,E_+,E_-\}$ of $\mathfrak{osp}(1|2)$ used in Pinczon's paper is slightly different. Namely, we have  $$x_{2\delta}=F,x_{-2\delta}=G, \frac{1}{2}h = Y,\quad \frac{1}{\sqrt{2}}x_{\delta}=E_+, \frac{1}{\sqrt{2}}x_{-\delta}=E_-.$$ 

In this paper we will focus on modules for which $C = -\frac{3}{16} \Id, \Omega = -\frac{1}{16} \Id$, see Lemma \ref{lem-cas}.

\subsection{Oscillator homomorphisms}
The Lie superalgebra $\mathfrak{osp}(1|2)$ can be embedded into $\mathcal D(1)$ in two natural ways:
$x_{\delta} \mapsto x, \; x_{\delta} \mapsto \partial$;
and 
$x_{\delta} \mapsto \partial, \; x_{\delta} \mapsto -x$. After adjusting the scalars in a way that is convenient for our setting, we consider the two  homomorphisms $\Phi_+$ and $\Phi_-$ from $U(\mathfrak{osp}(1|2))$ to $\mathcal D(1)$ defined as follows.
\[
\Phi_{+}:\quad
x_{\delta}\mapsto \tfrac{1}{\sqrt{2}}\,x,\ \ 
x_{-\delta}\mapsto \tfrac{1}{\sqrt{2}}\,\partial,\ \ 
h\mapsto x\partial+\tfrac{1}{2},\ \ 
x_{2\delta}\mapsto \tfrac{1}{2}\,x^{2},\ \ 
x_{-2\delta}\mapsto -\tfrac{1}{2}\,\partial^{2}.
\]
\[
\Phi_{-}:\quad
x_{\delta}\mapsto \tfrac{1}{\sqrt{2}}\,\partial,\ \ 
x_{-\delta}\mapsto -\tfrac{1}{\sqrt{2}}\,x,\ \ 
h\mapsto -(x\partial+\tfrac{1}{2}),\ \ 
x_{2\delta}\mapsto \tfrac{1}{2}\,\partial^{2},\ \ 
x_{-2\delta}\mapsto -\tfrac{1}{2}\,x^{2}.
\]

The homomorphism $\Phi_+$ with the exact same choice of constants was considered, for example, in \cite{GGF} (see  Proposition 4.1). Observe that we have $\Phi_- = \theta \Phi_+$ (hence, $\Phi_+ = -\theta \Phi_-$).

By direct computation one verifies the following.

\begin{lemma} \label{lem-cas}
The images of the Casimir elements $C$ and $\Omega$ under $\Phi_{\pm}$ are given by
$$\Phi_{\pm}(C) = -\frac{3}{16},  \quad\text{and}\quad \Phi_{\pm}(\Omega) = -\frac{1}{16} .$$
\end{lemma}
\begin{remark}
If we equip $\mathcal D(1)$ with the $\mathbb Z_2$-grading in which $x$ and $\partial$ are odd, then $\mathcal D(1)$ is the superalgebra denoted as $Cl(0|1)$ in \cite{GPS}; and as $A_{0|1}^+$ in \cite{HS}. Then, the algebra homomorphism $\Phi_{\pm}$ become a homomorphism of associative superalgebras. Since the exponential modules that we study are not $\Z_2$-graded in general (see \S \ref{sebsec:graded}), the $\mathbb Z_2$–grading of $\mathcal D(1)$ will not be used.
\end{remark}
\section{$F(X_+,X_-)$-realization of  $U(\h)$-free modules of rank $n$} \label{sec:FRn-matrix}

\subsection{The category of $U(\h)$-free modules of finite rank over $U(\mathfrak{osp} (1|2))$}

Let $n \in \mathbb{Z}_{>0}$. Define $\mathcal{FR}(n)$ to be the full subcategory of
$U(\mathfrak{osp}(1|2))\text{-mod}$ whose objects are
$U(\mathfrak{osp}(1|2))$-modules $M$ such that
$$
\Res^{U(\mathfrak{osp}(1|2))}_{U(\mathfrak{h})} M
\;\simeq\; U(\mathfrak{h})^{\oplus n}\; \simeq\; \C[h]^{\oplus n}.
$$
Morphisms in $\mathcal{FR}(n)$ are not necessarily parity-preserving.

Throughout the paper,  we will often identify the underlying space of any module $M$ in $\mathcal{FR}(n)$ with 
$M  =\C[h]^{\oplus n}$ and assume that the action of $h$ is by multiplication. In explicit terms, every $m \in M$ is  a column vector $$m = {\bf f}(h) = \bigl[f_1(h)\;\;\cdots\;\;f_n(h)\bigr]^{\mathsf T}\quad \text{with}\quad f_i(h) \in \C[h]$$
and 
$$h \; \cdot \; \bigl[f_1(h)\;\;\cdots\;\; f_{n}(h)\bigr]^{\mathsf T} \;=\; \bigl[hf_1(h)\;\;\cdots\;\;hf_{n}(h)\bigr]^{\mathsf T}.$$

\begin{proposition} \label{prop-mat-free}
Let  $M= \C[h]^{\oplus n}$ be in $\mathcal{FR}(n)$ such that  $h$ acts on the elements of $M$ by multiplication. Then  there exist matrices $X_+(h)$ and $X_-(h)$ in $\mbox{Mat}_n(\C[h])$ with  
\begin{equation}\label{eq-mat-free}
X_{+}(h)\,X_{-}(h-1)\;+\;X_{-}(h)\,X_{+}(h+1)\;=\;h\I_n 
\end{equation}
(in short, $X_+\sigma X_- + X_-\sigma^{-1} X_+ = h\I_n$).  such that the action of the generators $\mathfrak{osp}(1|2)$ on $M$ is given by 
$$
x_{\delta} \cdot {\bf f}(h) = X_+(h)\sigma ({\bf f}(h)),\;  x_{-\delta} \cdot {\bf f}(h) = X_-(h)\sigma^{-1} ({\bf f}(h)), \; {\bf f}(h) \in \C[h]^{\oplus n}.$$
(in short, $x_{\delta} \mapsto X_+\sigma, x_{-\delta} \mapsto X_-\sigma^{-1}$). In particular,
$$
x_{2\delta} \cdot {\bf f}(h) = X_+(h)X_+(h-1)\sigma^2 ({\bf f}(h)), \; x_{-2\delta} \cdot {\bf f}(h) = -X_-(h)X_-(h+1)\sigma^{-2} ({\bf f}(h)).$$

\end{proposition}
\begin{proof}
The matrices $X_{\pm}(h)$ are nothing else but the matrices of $x_{\pm \delta} \in \End(\C[h]^{\oplus n})$ in the standard basis. On the other hand, using $[h,x_{\delta}] = x_{\delta}$ one finds $x_{\delta}\cdot h{\bf f}(h) = (h-1)x_{\delta}\cdot {\bf f}(h)$. This easily leads to $x_{\delta} \cdot {\bf f}(h) = X_+(h)\sigma ({\bf f}(h))$.  The  proof of  $x_{-\delta} \cdot {\bf f}(h) = X_-(h)\sigma^{-1} ({\bf f}(h))$ uses the same  reasoning. The equation \eqref{eq-mat-free} follows from the relation $[x_{\delta}, x_{-\delta}]=h$. Since all relations of $\mathfrak{osp}(1|2)$ are  \eqref{rel-osp}, the proof is complete.
\end{proof}

\begin{definition} \label{def-x}
The module in $\mathcal{FR}(n)$ determined by the pair of matrices $(X_{+}(h), X_{+}(h))$ subject to the conditions \eqref{eq-mat-free} will be denoted by $F(X_+(h),X_-(h))$, or shortly by $F(X_+,X_-)$.
\end{definition}

\begin{remark}
The isomorphism class of $F(X_+,X_-)$ depends on the
$\GL_n(\C[h])$–action by $\sigma$-twisted conjugations:
\[
W(h)\cdot(X_+,X_-)
:=\bigl(W(h)^{-1}\,X_+(h)\,\sigma(W(h)),\; W(h)^{-1}\,X_-(h)\,\sigma^{-1}(W(h))\bigr).
\]
To be precise,  
\[
F(X_+,X_-)\simeq F(X'_+,X'_-)
\;\;\Longleftrightarrow\;\;
\exists\,W(h)\in GL_n(\C[h])\ \text{with}\ 
(X'_+,X'_-)=W(h)\cdot(X_+,X_-).
\]
This is precisely the change of $U(\h)$–basis on $\C[h]^{\oplus n}$ defined  by $v\mapsto W(h)\,v$.
\end{remark}

\section{Exponential modules of $U(\mathfrak{osp}(1|2))$ as $U(\mathfrak h)$-free modules} \label{sec:exp-mods}

\subsection{Exponential modules of $\mathcal D(1)$} Let $g(x) \in \C[x]$. The $\mathcal D(1)$-module with underlying space $E(g)= \C[x]\,e^{g(x)}$ and the natural $\mathcal D(1)$-action will be called the \emph{$g$-exponential module of $\mathcal D(1)$}. Equivalently, we can consider the module $E(g)$ with underlying space $\C[x]$ and action of the generators via 
$$
x \mapsto x, \; \partial \mapsto \partial + g'(x),
$$ 
i.e., $E(g) \simeq \mathcal O^{\epsilon_g}$. As mentioned earlier, we assume $g(0)=0$.

\subsection{Exponential modules of $U(\mathfrak{osp}(1|2))$} 

\begin{definition}\label{generalexp-module}
Let $g(x)\in\C [x]$.  The $\mathcal D(1)$-module $E(g) = \C[x]\,e^{g(x)}$ considered as a $ U(\mathfrak{osp}(1|2))$-module through $\Phi_{\pm}$ will be denoted by $E_{\pm}\left(g(x)\right)$ and will be referred as \emph{exponential modules} over $U(\mathfrak{osp}(1|2))$. Equivalently, we may consider $E_{\pm}\left(g(x)\right) = \C[x]$ with actions of the generators as follows
\[ E_{+}(g):\quad x_{\delta}\mapsto \tfrac{1}{\sqrt{2}}\,x,\ \ 
x_{-\delta}\mapsto \tfrac{1}{\sqrt{2}}\,(\partial + g'(x))\]
\[ E_{-}(g):\quad x_{\delta}\mapsto \tfrac{1}{\sqrt{2}}(\partial + g'(x)),\ \ 
x_{-\delta}\mapsto - \tfrac{1}{\sqrt{2}}\, x.\]

\end{definition}

\subsection{$\mathbb Z_2$-graded exponential modules} \label{sebsec:graded}
An $\mathfrak{osp}(1|2)$- module $M$ is called  an \emph{$\mathfrak{osp}(1|2)$-supermodule} if $M = M_{\bar{0}} \oplus M_{\bar{1}}$ and $x \cdot m \in M_{\alpha + \beta}$ if $x \in \mathfrak{osp}(1|2)_{\alpha}$ and $m \in M_{\beta}$. We similarly define the notion of a \emph{$\mathcal D(1)$-supermodule} using the $\mathbb Z_2$-grading of $\mathcal D(1)$ for which $\deg(x) = \deg(\partial) = \bar{1}$.
\begin{proposition}
The modules $E_{\pm}(g(x))$  are $\mathfrak{osp}(1|2)$-supermodules if and only if $g(x)$ is even, i.e. if $g(x)\in\C[x^{2}]$. In the latter case, 
    $$E_{\pm}(g(x))_{\bar 0}:=\C[x^{2}]\,e^{g(x)},\;
E_{\pm}(g(x))_{\bar 1}:=x\,\C[x^{2}]\,e^{g(x)}.$$
\end{proposition}

\begin{proof} We prove the equivalent statement that $M = \C[x]e^g$ is a super $\mathcal D(1)$-module if and only if  $g$ is even. Note that an existence of a  $\mathbb Z_2$-grading on $M$ compatible with the standard $\mathbb Z_2$-grading of $\mathcal D(1)$ is equivalent to an existence of a  linear involution $E: M\to M$ with
\begin{equation*}
E^2=1,\qquad Ex=-xE,\qquad E\partial=-\partial E.
\end{equation*}
We first prove the ``only if'' direction. Let $a(x)\in\C[x]$ be such that
$E(e^{g})=a(x)e^{g}$. Using $Ex=-xE$, we obtain
\begin{equation*}
E\bigl(f(x)e^{g}\bigr)=f(-x)\,a(x)\,e^{g}.
\end{equation*}
for all $f\in\C[x]$. Now applying   $E\partial=-\partial E$ on $e^{g(x)}$ and canceling $e^g$ leads to  the first-order ODE
\begin{equation*}
 a'(x)+a(x)\bigl(g'(x)+g'(-x)\bigr)=0.
\end{equation*}
Hence we have
\begin{equation*}
 a(x)=C\,\exp\bigl(\,g(-x)-g(x)\,\bigr)\qquad (C\in\C^*).
\end{equation*}
But  $a(x)\in\C[x]$ forces $g(-x)-g(x) = 0$.

   For the ``if'' direction, we define the involution $E$ by $E(f(x)e^g) = f(-x)e^g$. It is easy to check that $Ex=-xE$ and $E\partial=-\partial E$, and also, that this leads to the grading defined in the statement.
\end{proof}

\medskip

\subsection{Exponential modules as modules in $\mathcal{FR}(n)$}

\begin{proposition}\label{U(h)-decomp}
Let $g(x)$ be of degree $n\geq 1$.
Then
\[
E_{\pm}(g)=\bigoplus_{p=0}^{n-1}U(\mathfrak h)\cdot(x^{p}e^{g})
\]
as $U(\mathfrak h)$-modules. In particular, $E_{\pm}\left(g(x)\right)\in  \mathcal{FR}(n)$, i.e.,  $E_{\pm}\left(g(x)\right)$ are $U(\mathfrak{h})$-free modules of rank $n$ over $U\left(\mathfrak{osp}(1|2)\right)$.
\end{proposition} 

\begin{proof} We prove the statement for $E_+(g(x))$; the case  of $E_-(g(x))$ is analogous. Write $g(x)=\sum_{j=1}^n a_j x^j$ with $a_n\ne0$. Using $\Phi_+(h)=x\partial+\tfrac12$ and $\partial(e^{g})=g'(x)e^{g}$, for all $k\ge0$,
\[
 h\cdot\bigl(x^k e^{g}\bigr)=(k+\tfrac12)x^k e^{g}+\sum_{j=1}^n j a_j\,x^{k+j} e^{g}.
\]
Taking $k=\ell-n$ ($\ell\ge n$) and solving for $x^{\ell}e^{g}$ gives
\begin{equation}  \label{eq:red+} 
 x^{\ell}e^{g}=\frac{h-\ell+n-\tfrac12}{n a_n}\,x^{\ell-n}e^{g}-\sum_{j=1}^{n-1}\frac{j a_j}{n a_n}\,x^{\ell-n+j}e^{g},\qquad \ell\ge n.
\end{equation}

Hence,  by induction on $\ell\ge0$,  $x^{\ell}e^{g}$ is in the $U(\mathfrak h)$–span of $\{x^{p}e^{g}\}_{p=0}^{n-1}$. Thus
\[
E_{+}(g)=\sum_{p=0}^{n-1}U(\mathfrak h)\cdot(x^{p}e^{g}).
\]

 It remains to show that the sum is direct. Suppose $\sum_{p=0}^{n-1} f_p(h)\,x^{p}e^{g}=0$ with $f_p\in\C[h]$. Define $\deg_{x}\bigl(f(h)\,x^{p}e^{g}\bigr):=p+n\deg f$. Iterating \eqref{eq:red+} we see that the largest $x$–degree among all summands of $f_p(h)\,x^{p}e^{g}$ indeed equals $p+n\deg f_p$. Assume that there is at least one $p$ such that $f_p\neq 0$ and choose $r$ that maximizes $p+n\deg f_p$ among all $p$ with $f_p\ne0$. It is easy to see that such $r$ is unique. The term $f_r(h)\,x^{r}e^{g}$ contributes a nonzero multiple of $x^{D}e^{g}$ with $D=r+n\deg f_r$, while all other summands have largest $x$–degree strictly smaller than $D$. This leads to a contradiction, hence all $f_p=0$, and the sum is direct.
\end{proof}

\section{Simplicity and Isomorphism theorems}

\subsection{Simplicity}
\begin{proposition}\label{simplicity}

Let $\deg g \geq 1$.
Then $E_{\pm}\left(g\right)$ are simple $U\left(\mathfrak{osp}(1|2)\right)$-modules.
\end{proposition} 

\begin{proof} Let $g(x)=\sum_{j=1}^n a_j x^j$ with $a_n\neq 0$. 
Define
\[
D_s(g):=
\begin{cases}
\sqrt{2}\,x_{-\delta}\;-\;\displaystyle\sum_{j=1}^n j a_j\,(\sqrt{2}\,x_{\delta})^{\,j-1}, & s=+,\\[6pt]
\sqrt{2}\,x_{\delta}\;-\;\displaystyle\sum_{j=1}^n j a_j\,(-\sqrt{2}\,x_{-\delta})^{\,j-1}, & s=-.
\end{cases}
\]
In other words, we chose $D_s(g)$ so that 
\[
\Phi_+\bigl(D_+(g)\bigr)=\Phi_-\bigl(D_-(g)\bigr)=\partial-g'(x).
\]
For every $f\in\C[x]$ we have
\begin{equation}\label{eq:derivative}
D_s(g)\cdot\bigl(f(x)e^{g(x)}\bigr)=\bigl(\partial-g'(x)\bigr)\!\bigl(f(x)e^{g(x)}\bigr)=f'(x)\,e^{g(x)}.
\end{equation}
Then we apply a standard reduction-of-degree argument: 
if  $v=\sum_{\ell=0}^m b_\ell x^\ell e^g \in E_s(g)$ ($b_m\neq 0$), with repeated application of 
\eqref{eq:derivative} we obtain
\[
D_s(g)^{\,m}\cdot v=b_m\,m!\,e^{g}\neq 0,
\]
so $e^{g}\in U\!\left(\mathfrak{osp}(1|2)\right)\cdot v$.

On the other hand, $E_{+}(g)$ and $E_{-}(g)$ are generated by $e^{g}$ because: 
\[
f(x)\,e^{g}=f\!\bigl(\sqrt{2}\,x_{\delta}\bigr)\cdot e^{g} \mbox{ in } E_+(g),
\]
and 
\[
f(x)\,e^{g}=f\!\bigl(-\sqrt{2}\,x_{-\delta}\bigr)\cdot e^{g} \mbox{ in } E_-(g).
\]
Thus every nonzero submodule  of $E_{\pm}(g)$ contains $e^g$ and must coincide with $E_{\pm}(g)$. \end{proof}

\subsection{Isomorphism Theorem}
\begin{theorem}\label{thm:isom}
Let $g_1,g_2\in\C[x]$ be nonconstant polynomials, and let $s_1,s_2\in\{+,-\}$.
\begin{enumerate}
\item If $E_{s_1}(g_1)\simeq E_{s_2}(g_2)$, then $\deg g_1=\deg g_2$.
\item We have
\[
E_{s_1}(g_1)\simeq E_{s_1}(g_2)\quad\Longleftrightarrow\quad g_1-g_2\ \text{is constant}.
\]
(With the convention $g_i(0)=0$, this becomes $g_1=g_2$.)
\item   If $\deg g_1=\deg g_2\neq2$, then $E_{+}(g_1)\not\simeq E_{-}(g_2)$. If $a(x)=a_1x+a_2x^2$ and $b(x)=b_1x+b_2x^2$ with $a_2b_2\ne0$, then
\[
E_{+}(a)\;\simeq\;E_{-}(b)
\quad\Longleftrightarrow\quad
b_2=-\frac{1}{4a_2},\qquad b_1=-\frac{a_1}{2a_2}
\]
(i.e., if $e^{b(\xi)}$ is the integral Fourier transform of $e^{a(x)}$).
\end{enumerate}
\end{theorem}

\begin{proof}
(1) By Proposition \ref{U(h)-decomp}, both $E_{\pm}(g)$ are $U(\mathfrak h)$-free of rank $\deg g$. An $\mathfrak{osp}(1|2)$–isomorphism preserves this rank, so $\deg g_1=\deg g_2$.

\smallskip
(2) Assume $s_1=+$; the case $s_1=-$  is analogous.
Let $T:E_{+}(g_1)\to E_{+}(g_2)$ be an $\mathfrak{osp}(1|2)$–isomorphism.
Since in both modules $x_\delta=\tfrac1{\sqrt2}x$, we have $T x = x T$, if $f(x):=T(1)$, we have $T(p)=f\,p$ for all $p(x)$.
Intertwining with $x_{-\delta}$ yields to:
\[
T(\partial+g_1')=(\partial+g_2')T.
\]
Acting both sided on $p$ and after simplifying we obtain  the first-order ODE
\[
f'(x)+\bigl(g_2'(x)-g_1'(x)\bigr)f(x)=0.
\]

A nonzero polynomial solution $f$ exists iff $g_1'-g_2'\equiv0$, i.e.\ $g_1-g_2$ is constant. In the latter case $f$ is constant, and $T$ is an isomorphism. This proves (2).

\smallskip
(3) Suppose $T:E_{+}(g_1)\to E_{-}(g_2)$ is an $\mathfrak{osp}(1|2)$–isomorphism.
Set $B:=\partial+g_2'(x)$ and $B_1:=\partial+g_1'(x)$. Take again $f:=T(1)\ne0$ and use the relations
\[
T\,x=B\,T,\qquad T\,B_1=-x\,T.
\]
 By the first relation, $T(p)=p(B)\,f$ for all $p\in\C[x]$.
Applying the second relation to $x^k$ leads to
\[
\bigl(k B^{k-1}+g_1'(B)B^k\bigr)f \;=\; -x\,B^k f.
\]
Let $d=\deg g'_1 = \deg g'_2$. Then $\deg(B^k f)=\deg f+kd$ for $k\ge0$. If $d=0$ ($g_i$ are linear), the degreed on both sides differ by $1$. Hence $d\geq 1$. But the comparing degrees on both sides for large $k$ gives $d^2=1$, i.e.,\ $d=1$. Thus $\deg g_1=\deg g_2=2$. In this case for convenience we write $g_1=a$, $g_2=b$.

Since $g_1'(x)=2a_2 x + a_1$ and $g_2'(x)=2b_2 x + b_1$ with $a_2 b_2\ne0$, comparing the coefficients of the two top-degree terms (first on $x^{\deg f+k+1}$, then on $x^{\deg f+k}$) of
\[
\bigl(k B^{k-1}+g_1'(B)B^k\bigr)f \;=\; -x\,B^k f
\]
 yields
\[
4a_2 b_2=-1,\qquad b_1=2b_2 a_1,
\]
i.e.\ $b_2=-\frac{1}{4a_2}$ and $b_1=-\frac{a_1}{2a_2}$.

Lastly, assume $b_2=-\frac{1}{4a_2}$ and $b_1=-\frac{a_1}{2a_2}$, and let $a(x) = a_1x+a_2x^2$ and  $b(x) = b_1x+b_2x^2$. Then the map $T : E_{+}(g_1) \to E_{-}(g_2)$, $T(f(x)e^{a}) = f(\partial + b'(x))e^{b}$ defines an isomorphism of  $E_{+}(a)$ and $E_{-}(b)$.
\end{proof}

\subsection{The $F(X_+,X_-)$-realization of the exponential modules}
The following theorem provides an explicit realization
of the exponential modules $E_{\pm}(g)$
as modules $F(X_+,X_-)$ (recall Proposition \ref{prop-mat-free} and Definition \ref{def-x}).

\begin{theorem} \label{explicitrealization} If $g(x) = \sum_{i=1}^na_ix^i$ ($n\geq 1, a_n \neq 0$), then $E_{\pm}(g(x)) \simeq F(X_+(h), X_-(h))$, where the matrices $X_+(h), X_-(h)$ are defined as follows.
\begin{itemize}
\item[(i)] For $E_+(g(x))$,  we have  
    $$X_+(h) = \tfrac{1}{\sqrt{2}}\begin{bmatrix}
0      & 0      & \cdots & 0      & \tfrac{1}{na_n}\left(h-\tfrac{1}{2}\right) \\
1      & 0      & \cdots & 0      & -\tfrac{a_1}{na_n}              \\
0      & 1      & \ddots & \vdots & \vdots         \\
\vdots & \vdots & \ddots & 0      & -\tfrac{(n-2)a_{n-2}}{na_n}             \\
0      & 0      & \cdots & 1      & -\tfrac{(n-1)a_{n-1}}{na_n}
\end{bmatrix}, $$

$$X_{-}(h) = \tfrac{1}{\sqrt{2}}\begin{bmatrix}
 a_1 &  h+\tfrac{1}{2}  & 0 & \cdots & 0 \\
 2a_2 & 0 &  h+\tfrac{1}{2}  & \cdots & 0 \\
 \vdots & \vdots & \ddots & \ddots & \vdots \\
 (n-1)a_{n-1} & 0 & \cdots & 0 &  h+\tfrac{1}{2}  \\
 na_n & 0 & \cdots & 0 & 0
\end{bmatrix}.$$
\item[(ii)] For $E_-(g(x))$, we have 
$$X_+(h) = \tfrac{1}{\sqrt{2}}\begin{bmatrix}
 a_1 &  -\left(h-\tfrac{1}{2}\right)  & 0 & \cdots & 0 \\
 2a_2 & 0 &  -\left(h-\tfrac{1}{2}\right)  & \cdots & 0 \\
 \vdots & \vdots & \ddots & \ddots & \vdots \\
 (n-1)a_{n-1} & 0 & \cdots & 0 &  -\left(h-\tfrac{1}{2}\right)  \\
 na_n & 0 & \cdots & 0 & 0
\end{bmatrix}, $$
$$X_{-}(h) = - \tfrac{1}{\sqrt{2}}\begin{bmatrix}
0      & 0      & \cdots & 0      & -\tfrac{1}{na_n}\left(h+\tfrac{1}{2}\right) \\
1      & 0      & \cdots & 0      & -\tfrac{a_1}{na_n}              \\
0      & 1      & \ddots & \vdots & \vdots         \\
\vdots & \vdots & \ddots & 0      & -\tfrac{(n-2)a_{n-2}}{na_n}             \\
0      & 0      & \cdots & 1      & -\tfrac{(n-1)a_{n-1}}{na_n}
\end{bmatrix}.$$
\end{itemize}
\end{theorem}

In particular, if $s \in \{ +,-\}$, for $E_s(g)$,   $s\sqrt{2}\,X_{s}(h)$ is the companion matrix of the monic polynomial
\[
p_{s}(\lambda;h)\;=\;\lambda^{n}\;+\;\sum_{j=1}^{n-1}\frac{j\,a_{j}}{n\,a_{n}}\lambda^{j}
\;-\;\frac{sh-\tfrac12}{n\,a_{n}},
\]
i.e. $\det(\lambda \I-s\sqrt{2}\,X_{s}(h))=p_{s}(\lambda;h)$.

\begin{proof}
Recall from Proposition~\ref{U(h)-decomp}, 
$$E_{\pm}(g(x))\;=\;
\bigoplus_{l=0}^{n-1}
U(\mathfrak{h})\cdot\bigl(x^{l}e^{g(x)}\bigr).$$

Subsequently, the maps
$$
\Psi_{\pm}:E_{\pm}(g(x))\to \C[h]^{\oplus n}, \qquad  \sum_{l=0}^{n-1} f_l(h)\cdot\,\bigl(x^{l}e^{g(x)}\bigr) \mapsto \bigl[f_0(h)\;\;\cdots\;\;f_{n-1}(h)\bigr]^{\mathsf T},
$$
extend to isomorphisms of $U(\mathfrak{osp}(1|2))$-modules. We continue with the case of $E_{+}(g)$; the case of $E_-(g)$ is analogous. 

For \(0\le i\le n-2\),
\[
x_{\delta}\cdot(x^i e^{g})=\tfrac{1}{\sqrt2}\,x^{i+1}e^{g}.
\]

For \(i=n-1\), we use \eqref{eq:red+}:
$$x_{\delta}\cdot \left(x^{n-1}e^{g(x)}\right) = \frac{1}{\sqrt 2}x^{n}e^{g(x)} = \frac{1}{na_n\sqrt 2}\left(\left(h-\frac{1}{2}\right)\cdot e^{g(x)} - \sum_{j=1}^{n-1}ja_jx^je^{g(x)}\right).$$

For the action of $x_{-\delta}$, we have
\[
x_{-\delta}\cdot e^{g}=\tfrac{1}{\sqrt2}\,g'(x)\,e^{g}=\tfrac{1}{\sqrt2}\sum_{j=1}^n j a_j x^{j-1}e^{g},
\]
and, for \(1\le i\le n-1\),
\[
x_{-\delta}\cdot(x^i e^{g})
=\tfrac{1}{\sqrt2}\,\partial(x^i e^{g})
=\tfrac{1}{\sqrt2}\,(h+\tfrac{1}{2})\cdot \,(x^{i-1}e^{g}),
\]
since \(h=x\partial+\frac12\) implies \(\partial(x^i e^{g})=(h+\frac12)\,(x^{i-1}e^{g})\).
Therefore, on the $U(\mathfrak h)$-generating ordered set $\left\{x^{i}e^{g(x)}\right\}_{i=0}^{n-1}$, the matrices $X_+(h)$ and $X_-(h)$ are exactly the ones listed in part (i).
\end{proof}

\subsection{The rank-one case}
In this subsection we discuss $\mathfrak{osp}(1|2)$-modules that are $U(\mathfrak h)$-free of rank one. The rank-one $U(\mathfrak h)$-free modules of $\mathfrak{osp}(1|2m)$ are classified in \cite{CZ}. The result in \cite{CZ} for $m=1$ states that every $\mathfrak{osp}(1|2)$-module $M$ in $\mathcal{FR}(1)$ is isomorphic to $F(X_+,X_-)$, where $(X_+,X_-) = \left(\frac{1}{a\sqrt{2}}\left(h-\frac{1}{2}\right),\frac{a}{\sqrt{2}} \right)$ or $(X_+,X_-) = \left(\frac{a}{\sqrt{2}}, \frac{1}{a\sqrt{2}}\left(h+\frac{1}{2}\right)\right)$ for some $a \in \C^*$. But, by Theorem \ref{simplicity}, these are precisely the modules $E_{+}(ax)$ and $E_{-}(ax)$. Hence we have the following

\begin{corollary} Every module in $\mathcal{FR}(1)$ is isomorphic to either $E_{+}(ax)$ or $E_-(ax)$ for a unique 
$a\in \C^*$. 
\end{corollary}

In the rank-one case the isomorphisms 
$$
\Psi_{\pm}: E_{\pm}(ax) \to \mathbb C[h],\quad f(h)\cdot e^{ax} \mapsto f(h)
$$
are easy to describe:
\begin{proposition} \label{prop-psi-rank1}
For  $k\in \Z_{\geq 0}$, we have
$$
\Psi_+: x^ke^{ax}\mapsto \frac{k!}{a^k}\binom{h-1/2}{k},\qquad \Psi_-:  x^ke^{ax}\mapsto \frac{k!}{a^k}\binom{-h-1/2}{k}.
$$
Equivalently, we look at the coefficients of $t^k$ at
$$
\Psi_+: e^{tx+ax} \mapsto \left( 1+\frac{t}{a}\right)^{h-\frac{1}{2}}, \; \; \Psi_-:  e^{tx+ax} \mapsto \left( 1+\frac{t}{a}\right)^{-h-\frac{1}{2}}.
$$
\end{proposition}
\begin{proof}
By  Proposition \ref{U(h)-decomp},
there exist unique polynomials $P_k^{(\pm)}(h)\in\mathbb C[h]$ such that
\[
x^k e^{ax}=P_k^{(\pm)}(h)\cdot \,e^{ax},\qquad k\ge 0,
\]
and $\Psi_{\pm}(x^k e^{ax})=P_k^{(\pm)}(h)$.

Hence, for $k\ge 0$,
\begin{align*}
h\cdot(x^k e^{ax})
&=(x\partial+\tfrac12)\bigl(x^k e^{ax}\bigr)
 =\Bigl(k+\tfrac12\Bigr)x^k e^{ax}+a\,x^{k+1}e^{ax}, 
\end{align*}
which leads to
\[
x^{k+1}e^{ax}=\frac{h-\bigl(k+\tfrac12\bigr)}{a}\,x^{k}e^{ax}.
\]
This, together with the condition $P_0^{(+)}(h)=1$ leads to the desired formula for $\Psi_+$. The case of $\Psi_-$ is analogous.
\end{proof}

\begin{remark} \label{rem-ode}
One other way to define the isomorphisms $\Psi_i$ is by using the identity 
$$
\left(1+z\right)^{D}f(x) = f((1+z)x)
$$
for the Euler operator $D = x\partial$ and analytic functions $f(x)$. In fact, for $z=\frac{t}{a}$ and $f(x) = e^{ax}$, we simply have
$$
\Psi_{\pm}^{-1}: \left(1+\frac{t}{a}\right)^{D} \mapsto \left(1+\frac{t}{a}\right)^{D}e^{ax}, 
$$
where $D=\pm h-1/2$ on the left-hand side. Conversely, one may look for analytic functions $F(t;D)$ that satisfy
$$
e^{(t+a)x} = F(t;D)e^{ax}.
$$
Then after  differentiating by $t$ and using that 
\[
D\,e^{(a+t)x}=x\partial\bigl(e^{(a+t)x}\bigr)=(a+t)\,x\,e^{(a+t)x},
\] we obtain the differential equation
$$
 F_t'(t;D)=\frac{D}{a+t}\,F(t;D); \,  F(0;D)=1.
$$
The solution of this equation is precisely
$F(t,D) = \left(1+\frac{t}{a} \right)^D$. This approach will be discussed in the next subsection for arbitrary rank.
\end{remark}

\subsection{The isomorphisms $\Psi_{\pm}$ in terms of solutions of ODE} \label{subsec:ODE}
Inspired by Remark \ref{rem-ode},  we present the explicit form of the isomorphisms $\Psi_{\pm}$ in terms of solutions of ODE. This is achieved in the following proposition. 

\begin{lemma}
Fix $g(x)=\sum_{j=1}^{n}a_j x^{j}$ with $a_n\neq0$, and let $s\in\{+,-\}$.
Define the $\C[h]^{\oplus n}$–valued generating function
\[
F_s(t;h)\ :=\ \Psi_s\!\bigl(e^{tx}e^{g(x)}\bigr).
\]
Then each coordinate of $F_s(t;h)$ satisfies the scalar $n$th-order linear ODE
\begin{equation}\label{eq:master-ODE-short}
n a_n\,\partial_t^{\,n}F_s(t;h)
\;+\;\sum_{j=1}^{n-1} j a_j\,\partial_t^{\,j}F_s(t;h)
\;+\;t\,\partial_t F_s(t;h)
\;+\;\Bigl(\tfrac12 -s h\Bigr)F_s(t;h)\;=\;0.
\end{equation}

Moreover, let $F_{s,0},\dots,F_{s,n-1}$ be the unique $n$ solutions of
\eqref{eq:master-ODE-short} with  initial conditions
\begin{equation*}
\partial_t^{\,k}F_{s,j}(0;h)\;=\;\delta_{jk}\qquad (0\le j,k\le n-1).
\end{equation*}
Then
\[
\Psi_s\!\bigl(e^{tx}e^{g(x)}\bigr)
\;=\;\sum_{j=0}^{n-1}F_{s,j}(t;h)\,e_j,
\]
where $e_j$ is the $j$th standard basis vector of $\C[h]^{\oplus n}$. Equivalently,
\[
\Psi_s\!\bigl(x^{k}e^{g}\bigr)\;=\;\partial_t^{\,k}\Big|_{t=0}\Psi_s\!\bigl(e^{tx}e^{g}\bigr)
\;=\;e_k\qquad(k=0,1,\dots,n-1)
\]
with respect to the  $U(\mathfrak h)$–basis used in Theorem~\ref{explicitrealization}.
\end{lemma}
\begin{proof}
We use the following identities:
\begin{equation}\label{eq-iden}
\partial_t^{\,m}\!\bigl(e^{t x}e^{g}\bigr)=x^{m}e^{t x}e^{g},\qquad
(x\partial_x)\!\bigl(e^{t x}e^{g}\bigr)
=\Bigl(t\,\partial_t+\sum_{j=1}^{n}j a_j\,\partial_t^{\,j}\Bigr)\!\bigl(e^{t x}e^{g}\bigr).
\end{equation}
Recall the identity used in Proposition~\ref{U(h)-decomp}:
\[
n a_n x^{n}e^{g}+\sum_{j=1}^{n-1} j a_j x^{j}e^{g}=(x\partial_x)(e^{g}).
\]
Multiplying by $e^{t x}$ and using \eqref{eq-iden} gives
\[
\Bigl(n a_n\,\partial_t^{\,n}+\sum_{j=1}^{n-1} j a_j\,\partial_t^{\,j}+t\,\partial_t-x\partial_x\Bigr)\bigl(e^{t x}e^{g}\bigr)=0.
\]
We have
\[
x\partial_x=
\begin{cases}
\ h-\tfrac12& \text{on }E_{+}(g),\\
-\,h-\tfrac12& \text{on }E_{-}(g),
\end{cases}
\]
so applying the $\C[h]$–linear map $\Psi_{\pm}$ yields \eqref{eq:master-ODE-short}. Finally,
\(
\partial_t^{\,k}F_{\pm}(0;h)=\Psi_{\pm}\!\bigl(x^{k}e^{g}\bigr)
\)
since $\partial_t^{\,k}e^{t x}\big|_{t=0}=x^{k}$.
\end{proof}

\begin{example}[The rank-two case]
Take $n=2$ and $s=+$, and write $F,F_0,F_1$ for $F_+,F_{0,+},F_{1,+}$.
Then $F_r$ ($r=0,1$) satisfies
\[
2a_2\,F_r''(t;h)+(a_1+t)\,F_r'(t;h)+\Bigl(\tfrac12-h\Bigr)F_r(t;h)=0,
\qquad
\begin{cases}
F_0(0;h)=1,\ F_0'(0;h)=0,\\
F_1(0;h)=0,\ F_1'(0;h)=1.
\end{cases}
\]
With $z(t):=\dfrac{t+a_1}{\sqrt{2a_2}}$ and $S(t):=\dfrac{a_1 t+\tfrac12 t^2}{4a_2}$, the gauge change
$F_r(t;h)=e^{-S(t)}G_r(z(t))$ reduces the equation to the Weber (parabolic cylinder) form
\[
G_r''+\Bigl(\nu+\tfrac12-\tfrac{z^2}{4}\Bigr)G_r=0,\qquad \nu:=-h-\tfrac12,
\]
so
\[
F_r(t;h)=e^{-S(t)}\Bigl(\alpha_r\,D_{\nu}(z(t))+\beta_r\,D_{\nu}(-z(t))\Bigr),
\]
with $\alpha_r,\beta_r\in\C$ uniquely determined by the initial conditions. See, e.g.,  \cite[\S12]{DLMF}.
\end{example}

\subsection{The combinatorics of the isomorphisms $\Psi_{\pm}$} \label{subsec:combinatorics}
In this subsection, we give an explicit description of the isomorphism $\Psi_{\pm}$ in terms of recursive sequences. As usual, we fix $g(x)=\sum_{j=1}^{n}a_j x^{j}$ with $a_n\neq0$ and $n\geq 1$.

\begin{definition} Let \(s\in\{+,-\}\), $k\in \Z_{\geq 0}$, and $0\leq p\leq n-1$. Define $w^{(s)}_{k,p}(h)$ as follows. 
$$w^{(s)}_{q,p}(h) = \delta_{q,p}, \quad\text{for}\quad 0\leq q \leq n-1,$$
$$w^{(s)}_{l,p}(h) \;=\; \frac{sh-l+n-\frac{1}{2}}{n\,a_n}\,w^{(s)}_{l-n,p}(h)
\;-\;
\sum_{j=1}^{n-1}\frac{j\,a_j}{n\,a_n}\,
w^{(s)}_{\,l-n+j,p},\quad \text{for all} \quad l\geq n.$$
\end{definition}

\begin{lemma}
   We have the following formulas for   $\Psi_{\pm}: E_{\pm}(g) \to \mathbb C[h]^{\oplus n}$:
   
\begin{equation*}
    \Psi_{s}(
    x^ke^{g(x)})= [w^{(s)}_{k,0}(h)\;\;\cdots\;\; w^{(s)}_{k,n-1}(h)]^{\mathsf T},\quad \text{for all} \quad k\in \Z_{\geq 0}.
\end{equation*}
\end{lemma}

\begin{proof} We provide a summary of the important steps, and for brevity omit the technical part. 
By Proposition~\ref{U(h)-decomp},
$E_s(g)=\bigoplus_{p=0}^{n-1}U(\mathfrak h)\cdot (x^pe^g)$, hence for each $k\ge0$ there are unique $C^{(s)}_{k,p}(h)\in\C[h]$ with
\[
x^ke^g=\sum_{p=0}^{n-1} C^{(s)}_{k,p}(h)\,x^pe^g,
\qquad
C^{(s)}_{q,p}(h)=\delta_{q,p}\ (0\le q\le n-1).
\]

On the other hand,  using that $h=x\partial+\tfrac12$ on $E_+(g)$ and $h=-(x\partial+\tfrac12)$ on $E_-(g)$,
one obtains that for every $\ell\ge n$,
\[
x^\ell e^g
=\frac{(\pm h)-\ell+n-\frac12}{n a_n}\,x^{\ell-n}e^g
-\sum_{j=1}^{n-1}\frac{j a_j}{n a_n}\,x^{\ell-n+j}e^g,
\]
with the upper sign for $s=+$ and the lower sign for $s=-$. 

Combine the last two sets of identities and compare the
coefficients of the $U(\mathfrak h)$–basis $\{x^pe^g\}_{p=0}^{n-1}$. This shows that
the polynomials $C^{(s)}_{\ell,p}(h)$ (for $\ell\ge n$) satisfy exactly the
recurrence that defines $w^{(s)}_{\ell,p}(h)$, with the same initial data
$C^{(s)}_{q,p}=\delta_{q,p}$, $0\le q\le n-1$. By uniqueness,
$C^{(s)}_{k,p}(h)=w^{(s)}_{k,p}(h)$ for all $k\ge0$, $0\le p\le n-1$.\end{proof}
\begin{definition}
For $n\in\mathbb Z_{>0}$ and $M\in\mathbb Z_{\ge0}$, let $\Comp_n(M)$ denote the set
of ordered compositions of $M$ with parts in $\{1,\dots,n\}$.
We adopt the convention $\Comp_n(M)=\emptyset$ for $M\leq 0$.
\end{definition}

With the notations as above we have the following. (Empty products  are $1$, and empty sums are $0$.)
\begin{proposition}
For \(q\in\{1,\dots,n\}\) and \(r\in\mathbb Z\) set
\[
\phi_s(q;r):=
\begin{cases}
-\dfrac{(n-q)\,a_{\,n-q}}{n\,a_n}, & 1\le q\le n-1,\\
\dfrac{s\,h - r + n - \tfrac12}{n\,a_n}, & q=n.
\end{cases}
\]
For \(p\in\{0,1,\dots,n-1\}\) and \(m\in\mathbb Z_{\ge0}\), the polynomials
\(w^{(s)}_{m,p}(h)\in\mathbb C[h]\) defined by
\[
\Psi_s\!\bigl(x^m e^{g(x)}\bigr)
=\bigl[w^{(s)}_{m,0}(h),\dots,w^{(s)}_{m,n-1}(h)\bigr]^{\mathsf T}
\]
are given by
\[
w^{(s)}_{m,p}(h)=
\begin{cases}
\delta_{m,p}, & 0\le m\le n-1,\\[6pt]
\displaystyle\sum_{(\varepsilon_1,\dots,\varepsilon_r)\in\Comp_n(m-p)}\;
\prod_{t=1}^{r}
\phi_s\!\Bigl(\varepsilon_t\;;\;m-\sum_{u=1}^{t-1}\varepsilon_u\Bigr),
& \text{if }m\ge n,\ p\ge1,\\[16pt]
\displaystyle\sum_{\substack{(\varepsilon_1,\dots,\varepsilon_r)\in\Comp_n(m)\\ \varepsilon_r=n}}
\ \prod_{t=1}^{r}
\phi_s\!\Bigl(\varepsilon_t\;;\;m-\sum_{u=1}^{t-1}\varepsilon_u\Bigr),
& \text{if }m\ge n,\ p=0.
\end{cases}
\]
\end{proposition}

\begin{proof}
For \(0\le m\le n-1\), the statement is just the definition of \(\Psi_s\) on the
\(U(\mathfrak h)\)-basis \(\{x^p e^{g}\}_{p=0}^{n-1}\), giving \(w^{(s)}_{m,p}(h)=\delta_{m,p}\).

For \(m\ge n\), write the sum in the claimed formula by grouping compositions
\(\varepsilon\in\Comp_n(\,m-p\,)\) according to their first part \(q\in\{1,\dots,n\}\).
If the first part is \(q<n\), one obtains the term
\(-\dfrac{(n-q)a_{n-q}}{n a_n}\,w^{(s)}_{m-q,p}(h)\).
If the first part is \(q=n\), one obtains
\(\dfrac{s\,h-m+n-\tfrac12}{n a_n}\,w^{(s)}_{m-n,p}(h)\).
This reproduces the displayed recurrence after the change of index \(j=n-q\).

For \(p=0\) we additionally require the last part to be \(n\) (when \(m\ge n\));
this enforces the initial conditions \(w^{(s)}_{q,0}=\delta_{q,0}\) for \(0\le q\le n-1\)
while keeping the same recurrence. Hence the formula matches the unique solution
to the recurrence with the given initial data, i.e.\ it equals \(w^{(s)}_{m,p}(h)\).
\end{proof}

\begin{remark}
\begin{itemize}
\item[(i)] \emph{Walk interpretation.}
Fix \(m\ge n\) and \(p\in\{0,\dots,n-1\}\).
View a composition \(\varepsilon=(\varepsilon_1,\dots,\varepsilon_r)\) as a walk
descending from height \(m\) to \(p\) using steps \(\varepsilon_t\in\{1,\dots,n\}\).
Each step of size \(q<n\) contributes the constant weight
\(-\dfrac{(n-q)a_{n-q}}{n a_n}\).
Each step of size \(n\) taken when the current height is \(r\) contributes the ``affine'' weight
\(\dfrac{s\,h-r+n-\tfrac12}{n a_n}\).
When \(p=0\), the last step must have size \(n\).

\item[(ii)] \emph{Rank \(n=1\), i.e., \(g(x)=a_1x\).}
Here only the part \(q=1=n\) occurs, so for \(m\ge1\)
\[
w^{(s)}_{m,0}(h)
=\frac{1}{a_1^{\,m}}\prod_{t=1}^{m}\bigl(s\,h-m+t-\tfrac12\bigr)
=\frac{m!}{a_1^{\,m}}
\binom{s\,h-\tfrac12}{m}.
\]
recovering the formulas in Proposition \ref{prop-psi-rank1}.

\item[(iii)] \emph{Rank \(n=2\), i.e., \(g(x)=a_1x+a_2x^2\).}
Now steps have size \(1\) or \(2\), with weights
\(\phi_s(1;r)=-\dfrac{a_1}{2a_2}\) and \(\phi_s(2;r)=\dfrac{s\,h-r+\tfrac32}{2a_2}\).
For \(s=+\), \(p=0,1\), and $\varepsilon:=(\varepsilon_1,\dots,\varepsilon_r)$,
$$
w^{(+)}_{m,0}(h)
=\!\!\!\sum_{\substack{\varepsilon\in\Comp_2(m)\\ \varepsilon_r=2}}
\ \prod_{t=1}^{r}\phi_{+}\!\Bigl(\varepsilon_t\,; m-\sum_{u< t}\varepsilon_u\Bigr),\;\;
w^{(+)}_{m,1}(h)
=\!\!\sum_{\varepsilon\in\Comp_2(m-1)}
\ \prod_{t=1}^{r}\phi_{+}\!\Bigl(\varepsilon_t\,; m-\sum_{u< t}\varepsilon_u\Bigr).
$$
Take for example $s=+$,  $m=4$, and $p=0$. 
We sum over compositions $\varepsilon\in\Comp_2(4)$ whose last part is $2$:
these are $(2,2)$ and $(1,1,2)$. With the weights
$\phi_{+}(1;r)=-\dfrac{a_1}{2a_2}$ and $\phi_{+}(2;r)=\dfrac{h-r+\tfrac32}{2a_2}$, we find
\[
\begin{aligned}
w^{(+)}_{4,0}(h)
&=\ \underbrace{\phi_{+}(2;4)\,\phi_{+}(2;2)}_{\varepsilon=(2,2)}
\;+\;\underbrace{\phi_{+}(1;4)\,\phi_{+}(1;3)\,\phi_{+}(2;2)}_{\varepsilon=(1,1,2)}\\[4pt]
&=\ \left(\frac{1}{2a_2}\right)^2\!\left(h-\frac{5}{2}\right)\!\left(h-\frac{1}{2}\right)
\;+\;\left(-\frac{a_1}{2a_2}\right)^{2}\,\frac{1}{2a_2}\!\left(h-\frac{1}{2}\right).
\end{aligned}
\]

\end{itemize}
\end{remark}

\end{document}